\documentclass[11pt,a4paper]{article}
\usepackage[fleqn]{amsmath}
\usepackage{amsthm}
\usepackage{amssymb}
\usepackage{color}
\usepackage[section]{placeins}
\usepackage[colorlinks=true,linkcolor=blue,citecolor=blue]{hyperref}
\usepackage{enumerate}
\usepackage[margin=2.5cm]{geometry}
\usepackage{tikz,pgfplots}
\usetikzlibrary{decorations.markings}
\begin{document}
\tikzset{middlearrow/.style={
        decoration={markings,
            mark= at position 0.7 with {\arrow[scale=2]{#1}} ,
        },
        postaction={decorate}
    }
}
\newtheorem{theorem}{Theorem}
\newtheorem{lemma}[theorem]{Lemma}
\newtheorem{corollary}[theorem]{Corollary}
\newtheorem{proposition}[theorem]{Proposition}
\theoremstyle{remark}
\newtheorem*{remark}{Remark}
\newcommand{\TODO}[1]{\textcolor{red}{\textbf{TO DO:} {\small #1}}}
\newcommand{\NB}[1]{\textcolor{blue}{\textbf{N.B.} {\small #1}}}

\title{A revised Moore bound for mixed graphs}
\author{
	Dominique Buset\thanks{Universit\'e Libre de Bruxelles, Belgium}
	\and Mourad El Amiri\footnotemark[1]
	\and Grahame Erskine\thanks{Open University, Milton Keynes, UK}
	\and Mirka Miller\thanks{University of Newcastle, NSW, Australia}
	\and Hebert P\'erez-Ros\'es\thanks{University of Lleida, Spain}
}
\date{}
\maketitle
\let\thefootnote\relax
\let\thefootnote\relax
\begin{abstract}
The degree-diameter problem seeks to find the maximum possible order of a graph with a given (maximum) degree and diameter.
It is known that graphs attaining the maximum possible value (the \emph{Moore bound}) are extremely rare, but much activity is focussed on 
finding new examples of graphs or families of graph with orders approaching the bound as closely as possible.

There has been recent interest in this problem as it applies to \emph{mixed} graphs, in which we allow
some of the edges to be undirected and some directed. A 2008 paper of Nguyen and Miller derived an upper bound on the
possible number of vertices of such graphs. We show that for diameters larger than three, this bound can be reduced and we present 
a corrected Moore bound for mixed graphs, valid for all diameters and for all combinations of undirected and directed degrees.
\end{abstract}
\section{Introduction}
The Degree Diameter Problem for graphs has its motivation in the efficient design of interconnection networks. We seek to find the largest possible
order (number of vertices) of a graph given constraints on its diameter $k$ and the maximum degree $d$ of any of its vertices. Often the problem is studied
for undirected graphs, in which case the well-known \emph{Moore bound} (see e.g. \cite{miller2005moore}) states that the number of vertices cannot exceed
\begin{equation}\label{eq:undir}
M_{d,k}=
\begin{cases}
\displaystyle 1+d\frac{(d-1)^k-1}{d-2}&\text{ if }d>2\\
2k+1&\text{ if }d=2
\end{cases}
\end{equation}
In the undirected case it is well known \cite{miller2005moore} that no graph can achieve this bound if $k>2$, and for diameter 2 the only known \emph{Moore graphs}
correspond to degrees 2, 3 and 7, with the case $d=57$ being a famous open problem.

In the case of directed graphs the corresponding Moore bound \cite{miller2005moore} is given by 
\begin{equation}\label{eq:dir}
M_{d,k}=
\begin{cases}
\displaystyle \frac{d^{k+1}-1}{d-1}&\text{ if }d>1\\
k+1&\text{ if }d=1
\end{cases}
\end{equation}

In this paper we are concerned with the \emph{mixed} or \emph{partially directed} problem, 
in which we allow some of the edges in our graph to be directed and some undirected.
We conform to the most usual notation in the literature, 
so that the maximum undirected degree of a vertex (the number of undirected edges incident to it) is denoted by $r$.
The maximum directed degree is taken to mean the maximum number of out-arcs from any vertex and is denoted by $z$.
As usual we denote the diameter of a graph by $k$.

To bound the maximum possible number of vertices, the approach is to consider a spanning tree rooted at some arbitrary vertex.
It is not difficult to see that maximality is only achieved when each vertex has a unique parent at the previous level in the tree,
and the maximum possible number of neighbours at the next level. Figure \ref{fig:tree} shows such a tree for the case $z=3,r=3,k=2$.

Note that this Moore bound is only attained in a very small number of known cases.
Nguyen, Miller and Gimbert \cite{Nguyen2007} show that no graphs attaining the bound exist if the diameter $k\geq 3$.
For $k=2$, the known examples \cite{miller2005moore} are a family of Kautz graphs with $r=1,z\geq 1$ and a graph of Bos\'ak with $r=3,z=1$.
Recently, J\o rgensen \cite{Joergensen2015} has discovered a pair of graphs with $r=3,z=7$. 

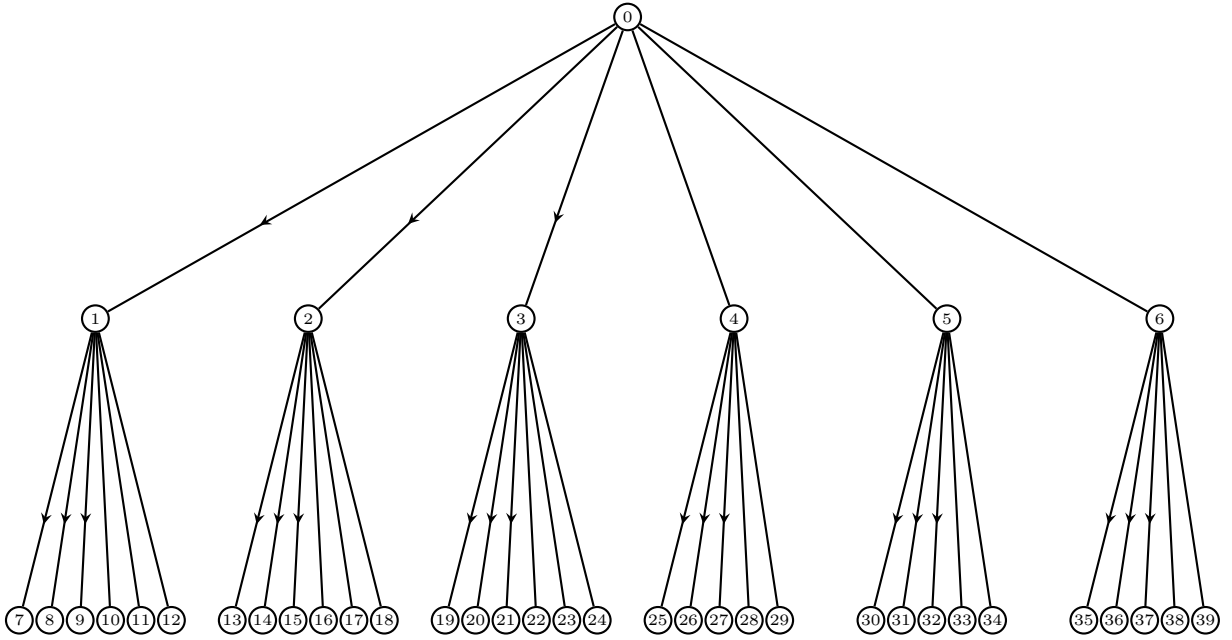
\begin{figure}
\centering
\begin{tikzpicture}[middlearrow=stealth,x=0.2mm,y=-0.2mm,inner sep=0.1mm,scale=2,
	thick,vertex/.style={circle,draw,minimum size=10,font=\tiny,fill=white},edge label/.style={fill=white}]
	\tiny
	\node at (200,0) [vertex] (v0) {$0$};
	\node at (25,100) [vertex] (v1) {$1$};
	\node at (95,100) [vertex] (v2) {$2$};
	\node at (165,100) [vertex] (v3) {$3$};
	\node at (235,100) [vertex] (v4) {$4$};
	\node at (305,100) [vertex] (v5) {$5$};
	\node at (375,100) [vertex] (v6) {$6$};
	\node at (0,200) [vertex] (v11) {$7$};
	\node at (10,200) [vertex] (v12) {$8$};
	\node at (20,200) [vertex] (v13) {$9$};
	\node at (30,200) [vertex] (v14) {$10$};
	\node at (40,200) [vertex] (v15) {$11$};
	\node at (50,200) [vertex] (v16) {$12$};
	\node at (70,200) [vertex] (v21) {$13$};
	\node at (80,200) [vertex] (v22) {$14$};
	\node at (90,200) [vertex] (v23) {$15$};
	\node at (100,200) [vertex] (v24) {$16$};
	\node at (110,200) [vertex] (v25) {$17$};
	\node at (120,200) [vertex] (v26) {$18$};
	\node at (140,200) [vertex] (v31) {$19$};
	\node at (150,200) [vertex] (v32) {$20$};
	\node at (160,200) [vertex] (v33) {$21$};
	\node at (170,200) [vertex] (v34) {$22$};
	\node at (180,200) [vertex] (v35) {$23$};
	\node at (190,200) [vertex] (v36) {$24$};
	\node at (210,200) [vertex] (v41) {$25$};
	\node at (220,200) [vertex] (v42) {$26$};
	\node at (230,200) [vertex] (v43) {$27$};
	\node at (240,200) [vertex] (v44) {$28$};
	\node at (250,200) [vertex] (v45) {$29$};
	\node at (280,200) [vertex] (v51) {$30$};
	\node at (290,200) [vertex] (v52) {$31$};
	\node at (300,200) [vertex] (v53) {$32$};
	\node at (310,200) [vertex] (v54) {$33$};
	\node at (320,200) [vertex] (v55) {$34$};
	\node at (350,200) [vertex] (v61) {$35$};
	\node at (360,200) [vertex] (v62) {$36$};
	\node at (370,200) [vertex] (v63) {$37$};
	\node at (380,200) [vertex] (v64) {$38$};
	\node at (390,200) [vertex] (v65) {$39$};

	\path
		(v0) edge [middlearrow] (v1)
		(v0) edge [middlearrow] (v2)
		(v0) edge [middlearrow] (v3)
		(v0) edge (v4)
		(v0) edge (v5)
		(v0) edge (v6)
		(v1) edge [middlearrow] (v11)
		(v1) edge [middlearrow] (v12)
		(v1) edge [middlearrow] (v13)
		(v1) edge (v14)
		(v1) edge (v15)
		(v1) edge (v16)
		(v2) edge [middlearrow] (v21)
		(v2) edge [middlearrow] (v22)
		(v2) edge [middlearrow] (v23)
		(v2) edge (v24)
		(v2) edge (v25)
		(v2) edge (v26)
		(v3) edge [middlearrow] (v31)
		(v3) edge [middlearrow] (v32)
		(v3) edge [middlearrow] (v33)
		(v3) edge (v34)
		(v3) edge (v35)
		(v3) edge (v36)
		(v4) edge [middlearrow] (v41)
		(v4) edge [middlearrow] (v42)
		(v4) edge [middlearrow] (v43)
		(v4) edge (v44)
		(v4) edge (v45)
		(v5) edge [middlearrow] (v51)
		(v5) edge [middlearrow] (v52)
		(v5) edge [middlearrow] (v53)
		(v5) edge (v54)
		(v5) edge (v55)
		(v6) edge [middlearrow] (v61)
		(v6) edge [middlearrow] (v62)
		(v6) edge [middlearrow] (v63)
		(v6) edge (v64)
		(v6) edge (v65)
		;
\end{tikzpicture}
\label{fig:tree}
\caption{The labelled Moore tree for $z=3,r=3,k=2$}
\end{figure}

In \cite{Nguyen2008} the general upper bound for the order of a graph with parameters $z,r,k$ is given as
\begin{equation}\label{eq:old}
M_{z,r,k}=1\,+\,(z+r)\,+\,z(z+r)+r(z+r-1)\,+\,\ldots\,+\,z(z+r)^{k-1}+r(z+r-1)^{k-1}
\end{equation}
It seems that this formula may have been extrapolated from the expressions for graphs of small diameter.
In this paper we develop a precise formula for the Moore bound, 
and show that for all diameters greater than 3 this is in fact smaller than the bound stated in \cite{Nguyen2008}.
\section{Revised Moore Bound}
\begin{theorem}\label{thm:moore}
Let $M_{z,r,k}$ denote the largest possible number of vertices in a mixed graph of diameter $k$, maximum directed degree $z$ and maximum undirected degree $r$.
Then:
\begin{equation}\label{eq:new}
M_{z,r,k}=A\frac{u_1^{k+1}-1}{u_1-1}+B\frac{u_2^{k+1}-1}{u_2-1}
\end{equation}
where:
\begin{align*}
v&=(z+r)^2+2(z-r)+1\\
u_1&=\frac{z+r-1-\sqrt{v}}{2}\\
u_2&=\frac{z+r-1+\sqrt{v}}{2}\\
A&=\frac{\sqrt{v}-(z+r+1)}{2\sqrt{v}}\\
B&=\frac{\sqrt{v}+(z+r+1)}{2\sqrt{v}}
\end{align*}
\end{theorem}
To prove the formula, we count vertices in the spanning tree by fixing an arbitrary vertex $w$ and consider the distance partition from $w$.
Denote by $L_j$ the maximum possible number of vertices in the graph at distance $j$ from $w$.
\begin{lemma}\label{lemma:distpart}
The maximum number of vertices in the distance partition satisfies the recurrence
\[L_j=(z+r-1)L_{j-1}+zL_{j-2};\quad L_0=1;\quad L_1=z+r\]
\end{lemma}
\begin{proof}
Clearly $L_0=1,L_1=z+r$. For $j\geq 2$ we proceed inductively. 
The key observation is that in a maximal graph, a vertex at level $j-1$ has exactly one parent at level $j-2$, 
but the number of its children at level $j$ depends on whether the edge from its parent is undirected or directed. 
If the edge is undirected then the vertex has at most $z+r-1$ children, and if it was directed then the vertex
has at most $z+r$ children, i.e. one more. Since the number of vertices at level $j-1$ with a directed edge from their parent is at most $zL_{j-2}$,
the recurrence follows.
\end{proof}
\begin{proof}[Proof of Theorem \ref{thm:moore}]
Clearly $\displaystyle M_{z,r,k}=\sum_{j=0}^k L_j$. 
To find an explicit form for $L_j$ we solve the second-order homogeneous recurrence defined by Lemma \ref{lemma:distpart}.
The characteristic equation of the recurrence system is $u^2+(1-z-r)u-z=0$. This has roots $u_1$ and $u_2$ as defined in the theorem.
Since $v=(z+r-1)^2+4z$, in all non-degenerate cases $v>0$ and so the roots are real and distinct.
From the general theory of second-order recurrences, the general solution of the system is $L_j=A u_1^j+B u_2^j$ where $A,B$ are constants
defined by the initial conditions. Elementary algebraic manipulation gives the values of $A,B$ as defined in the theorem.

Since $\displaystyle M_{z,r,k}=\sum_{j=0}^k L_j$, summing the geometric series for the $L_j$ gives the result
\[M_{z,r,k}=A\frac{u_1^{k+1}-1}{u_1-1}+B\frac{u_2^{k+1}-1}{u_2-1}\]

\end{proof}

Although the closed form solution \eqref{eq:new} looks rather more complex than the old version, 
we can show via straightforward algebraic manipulation that it generalises both the undirected and directed formulae.
\begin{proposition}~\par
\begin{enumerate}[(a)]
\item Setting $z=0$ in Equation \eqref{eq:new} recovers the undirected Moore bound \eqref{eq:undir}.
\item Setting $r=0$ in Equation \eqref{eq:new} recovers the directed Moore bound \eqref{eq:dir}.
\end{enumerate}
\end{proposition}

At first glance the formula \eqref{eq:new} offers little insight into the behaviour of the bound as $k$ increases.
However we can obtain a relatively straightforward estimate of its asymptotic behaviour.
\begin{proposition}\label{prop:nearest}
Suppose $r>0$. 
In the notation of Theorem \ref{thm:moore}, 
for sufficiently large $k$, $M_{z,r,k}$ is the nearest integer to $\displaystyle B\frac{u_2^{k+1}-1}{u_2-1}-\frac{A}{u_1-1}$.
\end{proposition}
\begin{proof}
It suffices to show that $|u_1|<1$. 

Now $\displaystyle 2\frac{\partial u_1}{\partial z}=1-\frac{z+r+1}{\sqrt{(z+r)^2+2(z-r)+1}}$. Since $r>0$, it follows that $(z+r+1)^2>(z+r)^2-2(z+r)+1$
and so $\displaystyle\frac{\partial u_1}{\partial z}<0$. So for any fixed $r>0$, $u_1$ is strictly decreasing as $z$ increases. 
When $z=0,u_1=0$ and for any $z$ we have $u_1>-1$ since $z+r+1-\sqrt{(z+r)^2+2(z-r)+1}>0$. Thus $0\geq u_1>-1$ for any $r>0$ and any $z\geq 0$.
\end{proof}

\section{Impact of revised bound on existing work}
We begin by offering some brief comparison between the old formula \eqref{eq:old} and our corrected version \eqref{eq:new}.
It is easy to verify that the formulae agree up to diameter 3, but our formula gives smaller results for all diameters $k>3$.
A comparison of values given by the old and new formulae for small $z,r$ is shown in the logarithmic plots in Figure~\ref{fig:tabs}.

\begin{figure}[h]

\begin{tikzpicture}[scale=0.9]
\begin{semilogyaxis}[
		every axis legend/.append style={nodes={right}},
    title={$z=1,r=1$},
    xlabel={Diameter},
    ylabel={Bound},
    legend pos=north west,
    ymajorgrids=true,
    grid style=dashed,
]
 
\addplot[mark=o]
    coordinates {
			(1,3)
			(2,6)
			(3,11)
			(4,20)
			(5,37)
			(6,70)
			(7,135)
			(8,264)
			(9,521)
			(10,1034)
			(11,2059)
			(12,4108)
			(13,8205)
			(14,16398)
			(15,32783)
			(16,65552)
			(17,131089)
			(18,262162)
			(19,524307)
			(20,1048596)
    };
 
\addplot[mark=x]
    coordinates {
			(1,3)
			(2,6)
			(3,11)
			(4,19)
			(5,32)
			(6,53)
			(7,87)
			(8,142)
			(9,231)
			(10,375)
			(11,608)
			(12,985)
			(13,1595)
			(14,2582)
			(15,4179)
			(16,6763)
			(17,10944)
			(18,17709)
			(19,28655)
			(20,46366)
    };
    \legend{Old formula,Corrected formula}
 
\end{semilogyaxis}
\end{tikzpicture}
\begin{tikzpicture}[scale=0.9]
\begin{semilogyaxis}[
		every axis legend/.append style={nodes={right}},
    title={$z=2,r=1$},
    xlabel={Diameter},
    ylabel={Bound},
    legend pos=north west,
    ymajorgrids=true,
    grid style=dashed,
]
 
\addplot[mark=o]
    coordinates {
			(1,4)
			(2,12)
			(3,34)
			(4,96)
			(5,274)
			(6,792)
			(7,2314)
			(8,6816)
			(9,20194)
			(10,60072)
			(11,179194)
			(12,535536)
			(13,1602514)
			(14,4799352)
			(15,14381674)
			(16,43112256)
			(17,129271234)
			(18,387682632)
			(19,1162785754)
			(20,3487832976)
    };
 
\addplot[mark=x]
    coordinates {
			(1,4)
			(2,12)
			(3,34)
			(4,94)
			(5,258)
			(6,706)
			(7,1930)
			(8,5274)
			(9,14410)
			(10,39370)
			(11,107562)
			(12,293866)
			(13,802858)
			(14,2193450)
			(15,5992618)
			(16,16372138)
			(17,44729514)
			(18,122203306)
			(19,333865642)
			(20,912137898)
    };
    \legend{Old formula,Corrected formula}
 
\end{semilogyaxis}
\end{tikzpicture}
\caption{Comparison of old Moore bound formula with corrected version}
\label{fig:tabs}
\end{figure}
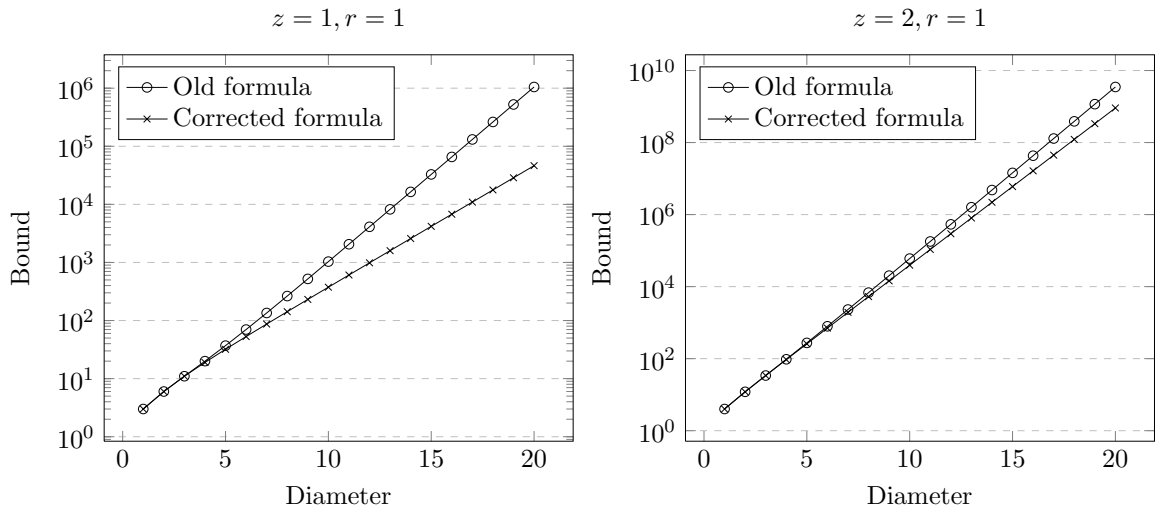

The old bound has been quoted by several authors, but none of them has used it in an essential way. 
For instance, \cite{Nguyen2007} proves that Moore graphs only exist for diameter 2. 
Recall that for diameter 2 the new upper bound agrees with the old one. 
The proof of the non-existence of Moore graphs for diameters greater than 2 does not make use of the upper bound, and hence is not affected by our correction. 
The papers \cite{Nacho2014} and \cite{Nacho2015} also focus on the case of diameter 2. 
In \cite{Nacho2014}, some conditions are given to construct a Moore graph by voltage assignment. 
On the other hand, \cite{Nacho2015} investigates properties of Moore graphs of directed degree one. Those results are also unaffected by our correction.  
\section{Acknowledgements}
The discussions leading to this paper arose during the workshop GraphMasters 12 hosted at the University of West Bohemia, Pilsen, Czech Republic in November 2014.
The remaining authors wish to dedicate this paper to the memory of our good friend and colleague Prof. Mirka Miller.
\bibliographystyle{amsplain}
\bibliography{MixedBound}
\end{document}